\documentclass[12pt]{amsart}

\usepackage{tgpagella}
\usepackage{euler}
\usepackage[T1]{fontenc}
\usepackage{amsmath, amssymb}
\usepackage[hidelinks]{hyperref}
\usepackage[english]{babel}
\usepackage{mathrsfs}
\usepackage{eucal}
\usepackage[all]{xy}
\usepackage{tikz}
\usepackage{todonotes}
\newtheorem{thm}{Theorem}[section]
\newtheorem*{thm*}{Theorem}
\newtheorem{lem}[thm]{Lemma}
\newtheorem{fact}[thm]{Fact}

\newtheorem{prop}[thm]{Proposition}
\newtheorem*{prop*}{Proposition}

\newtheorem{cor}[thm]{Corollary}
\newtheorem*{cor*}{Corollary}

\theoremstyle{definition}

\newtheorem*{defn*}{Definition}

\newtheorem{remark}[thm]{Remark}

\newtheorem*{question*}{Question}
\newtheorem*{Pquestion*}{Popa's question}

\newtheorem*{conv*}{Convention}

\newcommand{\N}{\mathbb{N}}

\newcommand{\norm}[1]{{\left\lVert #1\right\rVert}}

\def\bb{\mathbb}

\def\vp{\varphi}

\def\bb{\mathbb}

\def\sg{\sigma}

\def\cal{\mathcal}

\def\u{\mathsf 1}
\def\M{\mathcal{M}}

\makeatletter

\def\dotminussym#1#2{%
  \setbox0=\hbox{$\m@th#1-$}%
  \kern.5\wd0%
  \hbox to 0pt{\hss\hbox{$\m@th#1-$}\hss}%
  \raise.6\ht0\hbox to 0pt{\hss$\m@th#1.$\hss}%
  \kern.5\wd0}

\DeclareMathOperator{\supp}{supp}

\newcommand\cU{{\cal U}}

\def \u{\mathcal U}

\allowdisplaybreaks
\makeatletter
\def\l@subsection{\@tocline{2}{0pt}{2.5pc}{5pc}{}}
\def\l@subsubsection{\@tocline{2}{0pt}{5pc}{7.5pc}{}}
\makeatother

\begin{document}

%%%%%%%%%%%%%%%%%%%%%%%%%%%%%%%%%%%%%%%%%%%%%%

\title{Ultrapowers of spectral subspaces}

\author{Hiroshi Ando}
\address{}
\email{hiroando@math.s.chiba-u.ac.jp}
\thanks{Ando was partially supported by Japan Society for the Promotion of Sciences (JSPS) KAKENHI 25K07024}
\author{Isaac Goldbring}
\address{Department of Mathematics\\University of California, Irvine, 340 Rowland Hall (Bldg.\# 400),
Irvine, CA 92697-3875}
\email{isaac@math.uci.edu}
\urladdr{http://www.math.uci.edu/~isaac}
\thanks{Goldbring was partially supported by NSF grant DMS-2054477.}

\begin{abstract}
We prove, for any W$^*$-probability space $(M,\varphi)$ where $M$ is a type III$_1$ factor, any nontrivial, proper closed $F\subseteq \bb R$, and any nonprincipal ultrafilter $\u$ on $\bb N$, that the ultrapower $M(\sigma^\varphi,F)^\u$ of the spectral subspace $M(\sigma^\varphi,F)$ is a proper subset of the spectral subspace $M^\u(\sigma^{\varphi^\u},F)$.  We discuss the model-theoretic implications of this result.
\end{abstract}

\maketitle

\section{Introduction}

The main result of this paper is in response to the recent paper \cite{AGHS25} of Arulseelan, Hart, Sinclair, and the second author detailing a model-theoretic framework for \textbf{W$^*$-probability spaces}.  Let us begin with some context.

By a W$^*$-probability space we mean a pair $(M,\varphi)$ with $M$ a von Neumann algebra and $\varphi$ a faithful, normal state on $M$.  A model theoretic treatment of such pairs in the case that $\varphi$ is a trace was given in \cite{farah2014modelII}; a tracial von Neumann algebra $(M,\varphi)$ is viewed as a many-sorted structure whose sorts are given by operator norm bounded balls equipped with the metric stemming from the norm $\|x\|_\tau:=\sqrt{\varphi(x^*x)}$.  

Inspired by this framework, in \cite{dabrowski2019continuous} Dabrowski gave a model-theoretic treatment of an arbitrary W$^*$-probability space, again using operator norm bounded balls as the sorts, but this time equipped with the norm $$\|x\|_\varphi^\#: = \sqrt{\frac{\varphi(x^*x) + \varphi(xx^*)}{2}}$$ so as to have uniform continuity of the adjoint operation.  However, when $\varphi$ is not a trace, multiplication is not uniformly continuous when restricted to such a ball and thus cannot be included as a primitive in the language.  For that reason, Dabrowski works with ``smeared'' versions of multiplication that are uniformly continuous.

In \cite{AGHS25}, the authors wished to have honest multiplication as a primitive in the language and thus abandoned the use of operator norm bounded balls in favor of balls given by \textbf{totally bounded elements}.  Here, an element $a\in M$ is said to be \textbf{$K$-totally bounded} (with respect to $\varphi$) if both left and right multiplication by both $a$ and $a^*$ extend to bounded maps on $L^2(M,\varphi)$ of operator norm at most $K$ while $a$ is said to be totally bounded if it is $K$-totally bounded for some $K>0$.  The totally bounded elements are strongly dense in $M$, and thus using the $K$-totally bounded elements as sorts (as $K$ ranges over $\bb N$) ended up yielding a model-theoretic framework for W$^*$-probability algebras that allowed for multiplication in the language.  It is worth pointing out that the proof that this treatment actually works was quite nontrivial and used a number of technical ideas from Tomita-Takesaki theory as well as some ideas of Kadison.

In this paper, we suggest (and ultimately refute) an alternate proposal for a model-theoretic treatment of W$^*$-probability spaces, this time using \textbf{spectral subspaces} as sorts.  The precise definition of the spectral subspace $M(\sigma^\varphi,F)$ will be given in the next section, but suffice it to say that these spaces play a crucial role in the study of type III factors and the collection $\bigcup_{a>0}M(\sigma^\varphi,[-a,a])$ is strongly dense in $M$, whence at first glance it seems somewhat feasible that these sets could be used as sorts in a model-theoretic description of W$^*$-probability spaces.

A minimum requirement for an assignment $(M,\varphi)\mapsto S(M,\varphi)$ to be considered as a sort is that it is a \textbf{definable set}, which, operator-algebraically, means that $S(\prod_\u (M_i,\varphi))=\prod_\u S(M_i,\varphi)$ for any family $(M_i,\varphi_i)_{i\in I}$ of W$^*$-probability spaces and any ultrafilter $\u$ on $I$; here, and throughout this note, the ultraproduct is the \textbf{Ocneanu ultraproduct} \cite{AH14,Ocneanu85}. Our main result is that this commutation with ultraproducts does not generally hold in the case of spectral subspaces.  In fact, we show something stronger:

\begin{thm*}[Theorem \ref{prop Arvesonisundefinable}]
For any $\sigma$-finite type III$_1$ factor $M$, any faithful normal state $\varphi$ on $M$, any nonprincipal ultrafilter $\u$ on $\mathbb{N}$, and any nontrivial, proper closed subset $F\subsetneq \bb R$, we have that $M^\u(\sigma^{\varphi^\u},F)\not=M(\sigma^\varphi,F)^\u$. 
\end{thm*}

Here, $M(\sigma^{\varphi},F)^{\mathcal{U}}$ is the set of all $x\in M^{\mathcal{U}}$ which is represented by a bounded sequence $(x_n)_n$ such that $x_n\in M(\sigma^{\varphi},F)$ for every $n\in \N$. In particular, $$M(\sigma^{\varphi},\mathbb{R})^{\cU}=M^{\mathcal{U}} \text{ and }M(\sigma^{\varphi},\emptyset)^{\cU}=\{0\}.$$ 
In other words, except for the trivial situations $F=\emptyset, \bb R$, the spectral subspace $M(\sigma^\varphi,F)$ is not even a definable subset of $(M,\varphi)$ (let alone a definable set relative to the ambient theory of W$^*$-probability spaces).  We note that part of the main result of \cite{AGHS25} is that taking the set of $K$-totally bounded elements (for a fixed $K$) does commute with ultraproducts.  In some sense, our result validates the choice of totally bounded elements as choice of sorts. 
% \todo[inline]{(HA) Either this ``In some sense,..." or "Our result shows that despite..." should be removed, as they claim similar things. In any case the conclusion should be putted in the end of the paragraph. (IG) I actually think it looks good as it is.  I wouldn't remove either.}

There is a connection between spectral subspaces and totally bounded elements.  Indeed, in our recent preprint \cite{AG2026_bicentralizer}, we showed that if $x\in M(\sigma^\varphi,[-a,a])$ is a contraction, then $x$ is totally $e^{a/2}$-bounded.  It is easy to see that the converse may not hold even for a type I factor. Our result shows that despite the close relationship and similarity between these two classes of operators, totally bounded elements are preferable for the model theoretic approach to von Neumann algebra theory.
% As a corollary of our main result here, we see that the converse of this implication does not hold in general.

Although taking spectral subspaces does not commute with ultraproducts (or even ultrapowers), we show that there is always a containment in one direction:

\begin{thm*}[Proposition \ref{containment}]
For any family $(M_i,\varphi_i)_{i\in I}$ of W$^*$-probability spaces and any ultrafilter $\u$ on $I$, one has that $\prod_\u M_i(\sigma^{\varphi_i},F)\subseteq M(\sigma^\varphi,F)$, where $(M,\varphi):=\prod_\u (M_i,\varphi_i)$.
\end{thm*}

Model-theoretically, this implies that the spectral subspaces $M(\sigma^\varphi,F)$ form a \textbf{zeroset} relative to the theory of W$^*$-probability spaces.

To keep this note relatively short, we assume that the reader is familiar with basic facts about W$^*$-probability spaces and continuous model theory; they may consult our preprint \cite{AG2026_bicentralizer} for more information.  We do however review the necessary material on spectral subspaces in the next section.

This article was the result of the second author's three week visit to Chiba University as funded by a pilot Supplemental Research Collaboration Opportunity in Japan offered by The Division of Mathematical Sciences and the Office of International Science and Engineering of the National Science Foundation and the Japan Society for the Promotion of Science.  The second author would like to thank both organizations for the opportunity as well as to the first author and Chiba University for their hospitality during his visit.  

\section{Preliminaries}

\subsection{Background on spectral subspaces}

Here, we briefly recall Arveson spectral subspaces \cite{Arveson74}. More details can be found in  \cite[Chapter XI]{takesakiII}. 
We identify the dual group $\widehat{\mathbb{R}}$ of the additive group $\mathbb{R}$ with itself. For $f\in L^1(\mathbb{R})$, we define the Fourier transform $\hat{f}$  by 
\[\hat{f}(\lambda):=\int_{\mathbb{R}}e^{it\lambda}f(t)dt,\ \ \ \ \ \lambda \in \widehat{\mathbb{R}}=\mathbb{R}.\]

For $f\in L^1(\mathbb{R})$, we also consider the function $\sigma_f^\varphi:M\to M$ given by $$\sigma_f^{\varphi}(x):=\int_{\mathbb{R}}f(t)\sigma_t^{\varphi}(x)dt,$$ where the integral is taken in the $\sigma$-weak sense.  We note that $\|\sigma^\varphi_f(a)\|\leq \|f\|_1\|a\|$.

Below, we always let $M$ act on the Hilbert space $L^2(M,\varphi)$ associated with $\varphi$ with $\xi_{\varphi}$ the GNS vector, and $\Delta_{\varphi}$ is defined as the positive self-adjoint operator on $L^2(M,\varphi)$ as usual. 
For $f\in L^1(\mathbb R)$ and $x\in M$, we have 
\begin{equation}\sigma_f^{\varphi}(x)\xi_{\varphi}=\int_{\mathbb{R}}f(t)e^{it\log \Delta_{\varphi}}x\xi_{\varphi}dt=\hat{f}(\log \Delta_{\varphi})x\xi_{\varphi}.\label{eq modular op}
\end{equation}

For $x\in M$, $\text{Sp}_{\sigma^{\varphi}}(x)$ is defined by
\[\left \{\lambda \in \widehat{\mathbb{R}} \ : \ \hat{f}(\lambda)=0 \text{  for all } f\in L^1(\mathbb{R}) \text{  with  }\sigma^{\varphi}_f(x)=0\right \}.\] 

For a closed subset $E$ of $\widehat{\mathbb{R}}$, the \textbf{spectral subspace} of $\sigma^{\varphi}$ corresponding to $E$ is given by 
\[M(\sigma^{\varphi},E):=\{x\in M\ : \ \text{Sp}_{\sigma^{\varphi}}(x)\subseteq E\}.\]
From (\ref{eq modular op}) and spectral theory, it follows that  
\[x\in M(\sigma^{\varphi},E)\iff 1_E(\log \Delta_{\varphi})x\xi_{\varphi}=x\xi_{\varphi}.\]

The set $M_{\varphi}:=M(\sigma^{\varphi},\{0\})$ is the fixed point subalgebra of $M$ under $\sigma^{\varphi}$ and is called the \textbf{centralizer} of $\varphi$.
%\todo{After defining $\sigma_f$, mention $\|\sigma_f^\varphi(a)\|\leq \|f\|\cdot \|a\|$.}

% \begin{lem}\label{lem ArvesonspF}
% Let $M$ be a von Neumann algebra, let $\alpha=(\alpha_t)_{t\in\mathbb R}$ be a
% $\sigma$-weakly continuous one-parameter automorphism group on $M$, and let $F$ be a nonempty closed subset of $\mathbb{R}$. Then the following two conditions are equivalent. 
% \begin{itemize}
%     \item[{\rm (1)}] $x\in M^{\alpha}(F)$. 
%     \item[{\rm (2)}] $\alpha_f(x)=0$ for every $f\in L^1(\mathbb R)$ such that $\operatorname{supp}(\hat{f})\cap F=\emptyset$. 
% \end{itemize}
% \end{lem}

We will repeatedly use the following properties of the Arveson spectrum. 
\begin{fact}\label{fact spectralfacts}
Suppose that $(M,\varphi)$ is a W$^*$-probability space, $F\subseteq \bb R$ is a closed subset, $f\in L^1(\mathbb R)$, and $x\in M$.
\begin{enumerate}
    \item $M(\sigma^\varphi,-F)=M(\sigma^\varphi,F)^*$.
    \item $x\in M(\sigma^\varphi,F)$ if and only if $\sigma_g^\varphi(x)=0$ for every $g\in L^1(\mathbb R)$ such that $\operatorname{supp}(\hat{g})\cap F=\emptyset$. 
    \item If $\operatorname{supp}(\hat f)\subseteq F$, then $\sigma_f^\varphi(x)\in M(\sigma^\varphi,F)$.
    \item If $x\in M(\sigma^\varphi,F)$ and $\hat{f}=1$ on $F$, then $\sigma_f^\varphi(x)=x$.
    \item For $x\in M$ and $\lambda>0$, $x\in M(\sigma^{\varphi},\{\log \lambda\})$ if and only if $\varphi x=\lambda x\varphi$.
\end{enumerate}    
\end{fact}
We will also use the following fact (see \cite{AH14,katznelsonharmonic} for details on summability kernels and its use). 
\begin{fact}
For $a>0$, the {\bf Fej\'er kernel} $F_a:\mathbb{R}\to \mathbb{R}$ is defined by 
\[F_a(t):=\begin{cases}\dfrac{1-\cos (at)}{\pi at^2} & t\neq 0\\
\ \ \ a/2\pi & t=0. \end{cases}\]
Its Fourier transform is given by
\[\widehat{F_a}(\lambda)=\int_{\mathbb{R}}e^{it\lambda}F_a(t)dt=\begin{cases}1-\dfrac{|\lambda|}{a} & |\lambda|\le a\\
 \ \ \ 0 & |\lambda|>a.\end{cases}\]
It holds that $F_a\geq 0$ and $\|F_a\|_1=\widehat{F}_a(0)=1$.
In particular, we have $\sigma_{F_a}^{\varphi}(x)\in M(\sigma^{\varphi},[-a,a])$ for every $a>0$ and $x\in M$.

\begin{cor}\label{lem fourier}
    Let $I$ be a closed interval in $\mathbb R$ and let $s$ be an interior point of $I$. Then there exists $f\in L^1(\mathbb R)$ such that $\hat{f}(s)=1$, $\|f\|_1=1$, and $\operatorname{supp}(\hat{f})\subset I$. 
\end{cor}
\begin{proof}
Choose $r>0$ such that $[s-r,s+r]\subset I$, and set
\[
 f(t):=e^{-ist}F_r(t).
\]
Then
\[
 \widehat f(\lambda)=\widehat{F_r}(\lambda-s),
\]
so $\widehat f(s)=1$ and $\supp(\widehat f)\subset [s-r,s+r]\subset I$.  Since $\norm{F_r}_1=1$, $\norm f_1=1$ holds.
\end{proof}
\end{fact}
The next result follows from \cite[Lemma 4.13]{AH14}. 
\begin{fact}\label{fact spectralultrafact}
Suppose that $(M_i,\varphi_i)_{i\in I}$ is a family of W$^*$-probability spaces, $\u$ is an ultrafilter on $I$, and $F\subseteq \mathbb{R}$ is compact.  Suppose that, for each $i\in I$, $x_i\in M_i(\sigma^{\varphi_i},F)$ and $\sup_{i\in I}\|x_i\|<\infty$.  Then $(x_i)_{i\in I}$ represents an element $x=(x_i)_\u\in \prod_\u (M_i,\varphi_i)$.
\end{fact}

In the context of the previous fact, setting $(M,\varphi):=\prod_\u (M_i,\varphi_i)$, in the next section we will see that $x\in M(\sigma^\varphi,F)$.

The next lemma is a special case of \cite[Lemma 3.4]{AHHM20}:
\begin{lem}\label{lem isometry}
   Suppose that $(M,\varphi)$ is a W$^*$-probability space, where $M$ is a type III$_1$ factor.  Fix also a nonprincipal ultrafilter $\u$ on $\bb N$.  Then for every $b > 0$, there exists an isometry $v \in M^\mathcal{U}(\sigma^{\varphi^\mathcal{U}}, \{-b\})$.
\end{lem}
\begin{proof}
Consider $M^{\cU}$, which is a type III$_1$ factor with strictly homogeneous state space such that $(M^{\cU})_{\varphi^{\cU}}$ is a type II$_1$ factor \cite[Theorem 4.20 and Lemma 4.23]{AH14}. By \cite[Lemma 3.4]{AHHM20} with $\lambda=e^{-b}\le 1$, there exists $v'\in M^{\cU}$ such that $v'(v')^*=1$ and $v'\varphi^{\cU}=e^{-b}\varphi^{\cU}v'$. Therefore, by Fact \ref{fact spectralfacts} (1) and (5), $v:=(v')^*$ is an isometry belonging to $M^{\cU}(\sigma^{\varphi^{\cU}},\{-b\})$. 
\end{proof}

\subsection{Simple topological results}

\begin{lem}\label{lem boundary seq}
Suppose \(F\subseteq \mathbb R\) is closed and \(a\in \partial F\). Then at least one of the following holds:
\begin{enumerate}
\item[{\rm (1)}] there exists a strictly increasing sequence \((a_n)_{n\in \bb N}\) such that
\begin{itemize}
    \item $a_n<a$ for all $n\in \bb N$,
    \item $a_n\notin F$ for all $n\in \bb N$, and 
    \item $\lim a_n=a$
\end{itemize}
\item[{\rm (2)}] there exists a strictly decreasing sequence \((a_n)_{n\in \bb N}\) such that
\begin{itemize}
    \item $a_n>a$ for all $n\in \bb N$,
    \item $a_n\notin F$ for all $n\in \bb N$, and 
    \item $\lim a_n=a$
\end{itemize}
\end{enumerate}
\end{lem}

\begin{proof}
Assume, towards a contradiction, that neither \((1)\) nor \((2)\) holds.  Since (1) fails, we claim that then there exists \(\varepsilon>0\) such that
\[
(a-\varepsilon,a)\subseteq F.
\]
Indeed, if no such \(\varepsilon\) existed, then for every \(n\in \mathbb N\) we could choose a point
\[
x_n\in (a-1/n,a)\setminus F.
\]
From these points we construct a strictly increasing subsequence converging to \(a\). Choose \(n_1\in \mathbb N\) arbitrarily. Having chosen \(n_k\), choose \(n_{k+1}>n_k\) large enough so that
\[
a-\tfrac1{n_{k+1}} > x_{n_k}.
\]
Then pick
\[
x_{n_{k+1}}\in (a-1/n_{k+1},a)\setminus F.
\]
Because \(x_{n_{k+1}}>a-\tfrac1{n_{k+1}}>x_{n_k}\), the sequence \((x_{n_k})\) is strictly increasing, each term lies outside \(F\), and \(x_{n_k}\to a\) since
\[
a-\tfrac1{n_k}<x_{n_k}<a.
\]
This contradicts the assumption that \((1)\) fails. Hence the claim is proved.

By the same argument, if \((2)\) fails, then there exists \(\delta>0\) such that
\[
(a,a+\delta)\subseteq F.
\]

Therefore, if both \((1)\) and \((2)\) fail, there exist \(\varepsilon,\delta>0\) such that
\[
(a-\varepsilon,a)\subseteq F
\qquad\text{and}\qquad
(a,a+\delta)\subseteq F.
\]
Since also \(a\in \partial F\subseteq \overline F=F\), this implies
\[
(a-\varepsilon,a+\delta)\subseteq F,
\]
so \(a\) is an interior point of \(F\). However, this contradicts \(a\in \partial F\).

Thus at least one of \((1)\) or \((2)\) must hold.
\end{proof}

\begin{lem}\label{lem partial F contains nonzero}
Let \(F\subseteq \mathbb R\) be closed. Then the following are equivalent:
\begin{enumerate}
\item[{\rm (1)}] \(\partial F=\{0\}\).
\item[{\rm (2)}] \(F\in \bigl\{\{0\},\ [0,\infty),\ (-\infty,0]\bigr\}\).
\end{enumerate}
\end{lem}

\begin{proof}
\((2)\Rightarrow (1)\) is immediate. 
We prove \((1)\Rightarrow (2)\). Assume that $\partial F=\{0\}$. 
Since \(F\) is closed, we have \(\partial F\subseteq F\) and hence $0\in F$. 

Now let \(x\in F\setminus\{0\}\). Since \(x\notin \partial F\), we must have \(x\in \operatorname{int}(F)\). Thus every nonzero point of \(F\) is an interior point.

Therefore both
\[
F\cap (0,\infty)
\quad\text{and}\quad
F\cap (-\infty,0)
\]
are open and closed in the connected spaces \((0,\infty)\) and \((-\infty,0)\), respectively. Hence each of them is either empty or the whole interval. So each of the two possibilities
\[
F\cap (0,\infty)\in\{\emptyset,(0,\infty)\},
\qquad
F\cap (-\infty,0)\in\{\emptyset,(-\infty,0)\}
\]
must occur.

Since \(0\in F\), it follows that \(F\) is one of
\[
\{0\},\qquad [0,\infty),\qquad (-\infty,0],\qquad \mathbb R.
\]
But \(\partial F=\{0\}\) excludes \(F=\mathbb R\). Therefore $
F\in \bigl\{\{0\},\ [0,\infty),\ (-\infty,0]\bigr\}.$
\end{proof}

% \subsection{An analysis lemma}

% \todo[inline]{(HA) I added the calculation of F\'ejer kernel in ultraproduct check.tex, Prop 4.8. It is indeed hard to directly compute the Fourier transform of $F_a$. An easier route is to first compute the inverse Fourier transform of $\widehat{F}_a$ and check that it coincides with $F_a$. This kernel is explained in Katznelson's book. We can then consider the translated version $\psi_{b,a}$ as in Lemma 5.4 of BCnotes.tex to get \(\widehat{\psi_{b,a}}(a)=1\). If the interval is open, we can choose small \(b>0\) so that the Fourier transform \(\widehat{\psi}_{b,a}\) is supported in $[a-b,a+b]\subset I$.
% (IG) Great, thanks!  Can you insert it here?\\
% (HA) I moved this to Corollary \ref{lem fourier} in preliminary section and added the details.} 
% \begin{proof}
% \todo{Can you add the proof in here including details?  I wasn't able to verify that the integral was $1$.  Also, were you off by a minus sign somewhere?}
% \end{proof}

\section{Spectral subspaces are zerosets}

\begin{prop}\label{containment}
Suppose that $(M_i,\varphi_i)_{i\in I}$ is a family of W$^*$-probability spaces and $\u$ is an ultrafilter on $I$.  Fix also closed $F\subseteq \bb R$.  Set $(M,\varphi):=\prod_\u (M_i,\varphi_i)$.  Then $\prod_\u M_i(\sigma^{\varphi_i},F)\subseteq M(\sigma^\varphi,F)$.
\end{prop}

\begin{proof}
Take $x \in \prod_\cU M_i(\sigma^{\varphi_i}, F)$ and take a uniformly bounded sequence $(x_i)_{i \in I}\in \M^{\cU}(M_i,\varphi_i)$ such that $x=(x_i)_\u$ and $x_i \in M_i(\sigma^{\varphi_i}, F)$ for all $i\in I$. 
Let $f\in L^1(\mathbb{R})$ be such that $\operatorname{supp}(\hat{f})\cap F=\emptyset$. 
Then $\sigma_f^{\varphi_i}(x_i)=0$ for all $i\in I$ by Fact \ref{fact spectralfacts}(2), whence by \cite[Lemma 4.14]{AH14}, 
\[\sigma_f^{\varphi}(x)=(\sigma_f^{\varphi_i}(x_i))_{\cU}=0.\]
Then by Fact \ref{fact spectralfacts}(2) again, we obtain $x\in M(\sigma^\varphi,F)$.
\end{proof}

We refer the reader to \cite[Appendix B]{AG2026_bicentralizer} for background information on zerosets.

\begin{cor}
Let $T_{W^*}$ denote the theory of W$^*$-probability spaces.  Then for any closed subset $F\subseteq \bb R$, the $T_{W^*}$-functor that maps $(M,\varphi)$ to the (set of totally $1$-bounded elements of the) spectral subspace $M(\sigma^\varphi,F)$ is a $T_{W^*}$-zeroset. 
\end{cor}

\section{The main theorem}

\begin{thm}\label{prop Arvesonisundefinable}
    Let $(M,\varphi)$ be a W$^*$-probability space, where $M$ is a type III$_1$ factor. Let $\u$ be a nonprincipal ultrafilter on $\mathbb{N}$ and let $F$ be a closed, nonempty, proper subset of $\mathbb{R}$. Then $(M(\sigma^{\varphi},F))^{\cU}$ is a \emph{proper} subset of $M^{\cU}(\sigma^{\varphi^{\cU}},F)$. Consequently, the spectral subspace $M(\sigma^{\varphi},F)$ is a definable subset of $(M,\varphi)$ if and only if $F=\emptyset$ or $\bb R$. 
\end{thm}

\begin{proof}
Since $\bb R$ is connected and $F$ is closed, we have that $\partial F\neq \emptyset$. We first consider the case that $-a\in \partial F$ for some $a>0$. 
By Lemma \ref{lem boundary seq}, there exists a sequence $(a_n)_{n\in \mathbb{N}}$ of positive numbers such that $|a_n-a|<\frac{1}{n}$, $-a_n\notin F$ for each $n$, and either (i) $-a_1<-a_2<\dots<-a$ or (ii) $-a_1>-a_2>\dots>-a$ holds.  In either case, since $F$ is closed, for each $n$ we may find a closed interval $I_n$ of length less than $\frac{1}{n}$ such that $-a_n\in \operatorname{int}(I_n)$ and $I_n\cap F=\emptyset$. 

By Lemma \ref{lem isometry}, for each $n$, we may consider an isometry $v_n \in M^\mathcal{U}(\sigma^{\varphi^\mathcal{U}}, \{-a_n\})$. For each $n$, let $(\tilde{v}_{n,k})_{k\in \mathbb{N}}$ be a sequence such that $\|\tilde{v}_{n,k}\|\le 1$ for every $k,n$ and $v_n=(\tilde{v}_{n,k})_\u$. By Corollary \ref{lem fourier}, we may take $f_n\in L^1(\mathbb{R})$ such that: 
\begin{itemize}
    \item $\int_{\mathbb{R}}|f_n(t)|dt=1$
    \item $\hat{f_n}(-a_n)=1$
    \item $\operatorname{supp}(\hat{f_n})\subseteq I_n$
\end{itemize}

Set $v_{n,k}:=\sigma_{f_n}^{\varphi}(\tilde{v}_{n,k})$. Then $\|v_{n,k}\|\le \|f_n\|_1\|\tilde{v}_{n,k}\|\le 1$, and by Fact \ref{fact spectralfacts}(3), we have $v_{n,k}\in M(\sigma^{\varphi},I_n)$ for every $k,n$.  Moreover, since $\hat{f_n}(-a_n)=1$, by Fact \ref{fact spectralfacts}(4), we have 
\[v_{n}=\sigma_{f_n}^{\varphi^{\cU}}(v_{n})=(\sigma_{f_n}^{\varphi}(\tilde{v}_{n,k}))_{\cU}=(v_{n,k})_{\cU}.\]
Since $v_{n}$ is an isometry, we have 
$\displaystyle \lim_{k\to \cU}\|v_{n,k}\|_{\varphi}=1$. Therefore, for each $n$, we may find $k_n\in \mathbb{N}$ such that $w_n:=v_{n,k_n}$ satisfies:
\begin{itemize}
    \item $\|w_n\|\leq 1$
    \item $\|w_n\|_\varphi\geq 1-1/n$
    \item $w_n\in M(\sigma^\varphi,I_n)$.
\end{itemize}

Note that in case (i) we have  $I_n\subseteq [-a-\frac{2}{n},-a]$ for all $n\in \bb N$ and in case (ii) we have $I_n\subseteq [-a,-a+\frac{2}{n}]$ for every $n\in \mathbb{N}$. Thus, in either case, by Fact \ref{fact spectralultrafact}, we have that $w:=(w_n)_{\cU}$ is a well-defined element of $M^{\cU}$ such that $\|w\|_{\varphi^{\cU}}=1$. 
In case (i), we have $\operatorname{Sp}_{\sigma^{\varphi}}(w_m)\subseteq [-a-\frac{2}{n},-a]$ for every $m\ge n$, whence by redefining $w_m=0\,(m<n)$, we obtain  $\operatorname{Sp}_{\sigma^{\varphi^{\cU}}}(w)\subseteq [-a-\frac{2}{n},-a]$ by Proposition \ref{containment}. Since $n$ is arbitrary, we see that $w\in M^{\cU}(\sigma^{\varphi^{\cU}},\{-a\})\subseteq M^{\cU}(\sigma^{\varphi^{\cU}},F)$. By a similar argument, we also have $w\in M^{\cU}(\sigma^{\varphi^{\cU}},F)$ in case (ii) as well. 

We now show that $w\notin (M(\sigma^{\varphi},F))^{\cU}$. Assume, towards a contradiction, that  $w\in (M(\sigma^{\varphi},F))^{\cU}$. 
Then there is a bounded sequence $(z_n)_{n\in \mathbb{N}}$ such that $w=(z_n)_\u$ and such that $z_n\in M(\sigma^{\varphi},F)$ for every $n\in \mathbb{N}$. Then for every $n\in \mathbb{N}$, we have that
$$w_n \xi_\varphi \in 1_{I_n}(\log \Delta_{\varphi}) L^2(M, \varphi)$$
$$z_n \xi_\varphi \in 1_{F}(\log \Delta_{\varphi}) L^2(M, \varphi).$$
In particular, $z_n\xi_{\varphi}$ and $w_n\xi_{\varphi}$ belong to the spectral subspaces for $\log \Delta_{\varphi}$ corresponding to disjoint closed subsets. Thus, they are orthogonal in $L^2(M,\varphi)$. Since $w=(z_n)_{\cU}=(w_n)_{\cU}$, we have, by the Pythagorean law, that
\begin{align*}
0&=\lim_{n\to \cU}\|z_n\xi_{\varphi}-w_n\xi_{\varphi}\|^2\\
&=\lim_{n\to \cU}(\|z_n\xi_{\varphi}\|^2+\|w_n\xi_{\varphi}\|^2)\\
&\ge \lim_{n\to \cU}\|w_n\xi_{\varphi}\|^2=1,
\end{align*}
which is a contradiction. This concludes the proof for the case that $\partial F$ contains a negative real. 

If $a\in \partial F$ for some $a>0$, then $-a\in -\partial F=\partial (-F)$, and by the previous case, there is $w\in M^{\cU}(\sigma^{\varphi^{\cU}},-F)\setminus M(\sigma^{\varphi},-F)^{\cU}$. Then, by Fact \ref{fact spectralfacts}(1), we have that $v:=w^*\in M^{\cU}(\sigma^{\varphi^{\cU}},F)$. If $v\in M(\sigma^{\varphi},F)^{\cU}$, then $v=(v_n)_\u$ for some bounded sequence $(v_n)_{n\in \mathbb{N}}$ with $v_n\in M(\sigma^{\varphi},F)$, whence $v_n^*\in M(\sigma^{\varphi},-F)$ for every $n$, which implies $w=(v_n^*)_{\cU}\in M(\sigma^{\varphi},-F)^{\cU}$, a contradiction. 

By Lemma \ref{lem partial F contains nonzero}, it remains to consider the cases $F=\{0\}, (-\infty,0]$, or $[0,\infty)$.  However, these cases can be handled in a similar way.  
For example, for \(F=\{0\}\) or \(F=[0,\infty)\), choose
\[
 I_n:=\left[-\frac{3}{2n},-\frac{1}{2n}\right].
\]
Then \(-1/n\in \operatorname{int}(I_n)\), \(I_n\cap F=\emptyset\), and the same construction using isometries
\[
 v_n\in M^{\cU}(\sg^{\vp^{\cU}},\{-1/n\})
\]
produces \(w\in M^{\cU}(\sg^{\vp^{\cU}},\{0\})\subseteq M^{\cU}(\sg^{\vp^{\cU}},F)\), while the same orthogonality argument gives
\[
 w\notin M(\sg^\vp,F)^{\cU}.
\]
Then the case \(F=(-\infty,0]\) follows by taking adjoints from the case \([0,\infty)\).
% If $F=\{0\}$ or $[0,\infty)$, then argue as above using an isometry $v_n\in M^{\cU}(\sigma^{\varphi^{\cU}},\{-\frac{1}{n}\})$ to get $w\in M^{\cU}(\sigma^{\varphi^{\cU}},F)\setminus M(\sigma^{\varphi},F)^{\cU}$.  By taking adjoints as in the previous paragraph, one can also conclude the result for the case $F=(-\infty,0]$.
\end{proof}

\begin{remark}
The strict inclusion $M(\sigma^\varphi,\{0\})^\u\subsetneq M^\u(\sigma^{\varphi^\u},\{0\})$ was observed by Haagerup and the first author in \cite[Remark 4.28]{AH14} in the case that $\varphi$ is an \textbf{ergodic} state, meaning that $M_\varphi=\bb C$.  (By a recent result of Marrakchi and Vaes \cite{marrakchivaes2024ergodic}, every separable type III$_1$ factor $M$ admits such an ergodic state; in fact, ergodic states are comeager in the space of all faithful normal states on $M$.)  The previous theorem is thus new in the case that $F=\{0\}$ but $\varphi$ is a non-ergodic state on $M$.  
\end{remark}

\bibliographystyle{siam}
\bibliography{reference}

\end{document}